\documentclass [12pt,twoside,a4paper]{article}
\usepackage{amsfonts}
\usepackage{amsthm}
\usepackage{amsmath}
\usepackage{amstext}
\usepackage{amssymb}
\usepackage{mathrsfs}
\usepackage{amscd}
\usepackage{xypic}
\usepackage{cite}
\usepackage[numbers,sort&compress]{natbib}
\usepackage{epsf}              
\usepackage{fancybox}          
\usepackage{color}             
\usepackage{fancyhdr}
\usepackage[hang,footnotesize]{caption2}  
\usepackage{enumerate} 
\usepackage{float}
\usepackage{makecell}
\usepackage{subfigure}
\usepackage{graphicx}
\usepackage{epstopdf}
\usepackage{epsfig}
\usepackage[numbers]{natbib}
\usepackage{ragged2e}
\usepackage{graphicx,epsfig,epstopdf,subfigure}
\numberwithin{equation}{section}
\newfont{\aaa}{cmb10 at 19pt}
\newfont{\bbb}{cmb10 at 14pt}
\newtheorem{theorem}{Theorem}[section]
\newtheorem{corollary}[theorem]{Corollary}
\newtheorem{lemma}[theorem]{Lemma}
\newtheorem{claim}[theorem]{Claim}

\newtheorem{problem}[theorem]{Problem}
\newtheorem{conjecture}[theorem]{Conjecture}

\makeatletter

\newcommand{\Rmnum}[1]{\expandafter\@slowromancap\romannumeral #1@}
\makeatother

\newcommand{\beq}{\begin{equation}}
\newcommand{\eeq}{\end{equation}}
\newcommand{\bey}{\begin{eqnarray}}
\newcommand{\eey}{\end{eqnarray}}
\newcommand{\beyy}{\begin{eqnarray*}}
\newcommand{\eeyy}{\end{eqnarray*}}

\bfseries\normalfont

\begin{document}

\title{A generalization of the Hamiltonian cycle in dense digraphs \thanks{Zhilan Wang and Jin Yan are supported by NNSF of China (No.12071260).}}

\author{Jie Zhang\textsuperscript{1}, Zhilan Wang\textsuperscript{2}\thanks{Corresponding author: E-mail address: 202220319@mail.sdu.edu.cn.}, Jin Yan\textsuperscript{2}\unskip\\[2mm]
\textsuperscript{1} {\footnotesize Xi'an Research Institute of High-tech Hongqing Town, Xi'an, Shanxi 710025, China}
\\
\textsuperscript{2} {\footnotesize School of Mathematics, Shandong University, Jinan 250100, China}}
\date{}
\maketitle

%

\begin{abstract}
Let $D$ be a digraph and $C$ be a cycle in $D$. For any two vertices $x$ and $y$ in $D$, the distance from $x$ to $y$ is the minimum length of a path from $x$ to $y$. We denote the square of the cycle $C$ to be the graph whose vertex set is $V(C)$ and for distinct vertices $x$ and $y$ in $C$, there is an arc from $x$ to $y$ if and only if the distance from $x$ to $y$ in $C$ is at most $2$. The reverse square of the cycle $C$ is the digraph with the same vertex set as $C$, and the arc set $A(C)\cup \{yx: \mbox{the vertices}\ x, y\in V(C)\ \mbox{and the distance from $x$ to $y$ on $C$ is $2$}\}$. In this paper, we show that for any real number $\gamma>0$ there exists a constant $n_0=n_0(\gamma)$, such that every digraph on $n\geq n_0$ vertices with the minimum in- and out-degree at least $(2/3+\gamma)n$ contains the reverse square of a Hamiltonian cycle. Our result extends a result of Czygrinow, Kierstead and Molla.
\end{abstract}
absorption method

\section{Introduction}
Hamiltonicity is one of the most central notions in graph theory, and it has been extensively studied by numerous researchers. 
One of the most notable examples here is Dirac's theorem \cite{Dirac} from $1952$, which states that a graph with $n\geq3$ vertices and minimum degree at least $n/2$ contains a Hamiltonian cycle, which is a cycle passing through every vertex of the graph exactly once. This result has an impact on extremal graph theory, leading to a wide of range of results usually known as Dirac-type results. With the development of sophisticated embedding techniques in recent decades, the Dirac problem is nowadays well-understood for a large class of graphs.

To clarify, the \emph{$k$-th power} of a directed path $P_l=v_0\cdots v_l$ refers to the directed graph $P_l^k$ with the same set of vertices as $P_l$, and in $P_l^k$, there exists an edge $v_iv_j$ if and only if $i<j\leq i+k$. Similarly, the $k$-th power of a directed cycle is defined in the same manner. A natural and more general property is to contain the $k$-th power of a Hamiltonian cycle. Extending Dirac's theorem, in $1962$ P\'{o}sa \cite{Posa} conjectured that every graph $G$ on $n$ vertices with the minimum degree at least $\frac{2n}{3}$ contains the square of a Hamiltonian cycle. Furthermore, Seymour \cite{Seymour} conjectured in $1974$ that for any integer $k$, the minimum degree bound for a graph to contain the $k$-th power of a Hamiltonian cycle is $\frac{kn}{k+1}$. After two decades and several papers on this question, Koml\'{o}s, S\'{a}rk\"{o}zy and Szmer\'{e}di \cite{Komlos} confirmed Seymour's conjecture.

\smallskip

Particularly, one can ask similar questions for digraphs, which often poses greater challenges. For the special digraph: an \emph{oriented digraph} $G$, which is an orientation of a simple graph, Thomassen \cite{Thomassen} posed the question of determining the \emph{minimum semi-degree $\delta^0(G)$}, which is the minimum of the minimum out-degree and the minimum in-degree of $G$, to guarantee the existence of a Hamiltonian cycle in $G$. This was answered by Keevash, K\"{u}hn and Osthus \cite{Keevash}, who showed that $\delta^0(G)\geq \frac{3n-4}{8}$ is sufficient to force a Hamiltonian cycle in a sufficient large oriented graph $G$. However, for the problem for the square of Hamiltonian cycles in the oriented graph $G$, it is not well resolved. Treglown \cite{Treglown} asserts that $\delta^0(G)\geq\frac{5n}{12}$ is necessary, which was subsequently improved by DeBiasio (personal communication). He stated that $\delta^0(G)\geq\frac{3n}{7}-1$ is needed by using a slightly unbalanced blow up of the Paley tournament on seven vertices. It would be intriguing to determine, even in an asymptotic sense, the optimal value of $\delta^0(G)$ that guarantees the existence of the square of a Hamiltonian cycle.

\smallskip

For general digraphs, Ghouila-Houri \cite{Ghouila} proved that any digraph $D$ on $n\geq2$ vertices with $\delta^0(D)\geq n/2$ has a Hamiltonian cycle. Unfortunately, no progress has been made on the $k$-th power of a Hamiltonian cycle problem of digraphs. In this paper, we give the following result. Let $D$ be a digraph and $C$ be a cycle in $D$. The \emph{reverse square} of the cycle $C$ is the digraph that has the same vertex set as $C$, and arc set to be the union of all arcs of $C$ and all arcs $yx$ satisfying the distance from $x$ to $y$ is $2$ on $C$.
\begin{theorem}\label{main}
For any real number $\gamma>0$ there exists a constant $n_0=n_0(\gamma)$ such that every digraph $D$ on $n\geq n_0$ vertices with $\delta^0 (D) \geq (2/3+\gamma)n$ contains the reverse square of a Hamiltonian cycle.
\end{theorem}
The following result can be reduced from Theorem \ref{main}, which also is proved by Czygrinow, Kierstead and Molla \cite{Czygrinow}.
\begin{corollary}
For any real number $\gamma>0$ there exists a constant $n_0=n_0(\gamma)$ such that every digraph $D$ on $n\geq n_0$ vertices with $\delta^0 (D) \geq (2/3+\gamma)n$ contains $\lfloor \frac{|V(D)|}{3}\rfloor$ disjoint triangles.
\end{corollary}
Also, due to the difficulty of these problems in general, it is natural to ask what happens in \emph{tournaments}, which are orientations of complete graphs. It is a well-known result that every tournament with the minimum semi-degree at least $\frac{n-2}{4}$ has a Hamiltonian cycle, and this lower bound is optimal. Further, Bollob\'{a}s and H\"{a}ggkvist \cite{Bollobas} proved that for every $\varepsilon>0$ and $k$, there exists an $n_0=n_0(\varepsilon, k)$ such that every tournament $T$ on $n\geq n_0$ vertices with $\delta^0(T)\geq \frac{n}{4}+\varepsilon n$ contains the $k$-th power of a Hamiltonian cycle. Lately, Dragani\'{c}, Correia and Sudakov \cite{Draganic} refined the additive error in the degree condition, who proved that there exists a constant $c=c(k)>0$ such that any tournament $T$ on $n$ vertices with $\delta^0(T)\geq\frac{n}{4}+cn^{1-1/\lceil k/2\rceil}$ contains the $k$-th power of a Hamiltonian cycle. In particular, they also showed that a constant error term is enough for the tournament to
contain the square of a Hamiltonian cycle.

\smallskip

\textbf{Organization.} The rest of the paper is organised as follows. Our approach of the proof of Theorem \ref{main} uses a hybrid of the Regularity-Blow-up method and the Connecting-Absorbing method. The absorption method that was introduced by R\"{o}dl, Ruci\'{n}ski, and Szemer\'{e}di \cite{Rodl}. However, we need to adapt these ideas to the reverse square of Hamiltonian cycles in digraphs instead of tight cycles in hypergraphs. In Section $2$, we begin by presenting relevant notations and some useful results. Moving on to Section $3$,
we first introduce the main tools, namely Connecting Lemma, Absorbing Lemma and Path-Covering Lemma. Then we give the proof of Theorem \ref{main}. Finally, Section $4$ contains some concluding remarks to wrap up the paper.

\section{Preparations for Theorem \ref{main}}
\subsection{Definitions and notations}
For notations not defined in this paper, we refer the readers to \cite{Bang-Jensen3}. We denote by $G(A, B)$ a bipartite graph $G$ with vertex classes $A$ and $B$, and $e_G(A, B)$ is the number of edges between $A$ and $B$.
Let $D=(V, A)$ be a digraph. The cardinality of a vertex set $X\subseteq V$ is denoted by $|X|$, and we call $X$ to be a \emph{$i$-set} if $|X|=i$. The subdigraph of $D$ induced by $X$ is denoted as $D[X]$. Let $D-X=D[V\setminus X]$. Let $Y$ be another subset of $V$ which is disjoint with $X$. We define the arcs from $X$ to $Y$ to be \emph{$X-Y$ edges}. The \emph{out-neighbourhood} (resp., \emph{in-neighbourhood}) of a vertex $v$ in $D$ is defined as $N^{+}(v)=\{u: vu\in A\}$ (resp., $N^{-}(v)=\{w: wv\in A\}$). The \emph{out-degree} (resp., \emph{in-degree}) of $v$ in $D$, which is denoted by $d^+(v)$ (resp. $d^-(v)$), is the cardinality of $N^{+}(v)$ (resp., $N^{-}(v)$), that is, $d^{+}(v)=|N^{+}(v)|$ (resp., $d^{-}(v)=|N^{-}(v)|$). The \emph{minimum out-degree} $\delta^+(D)=\min\{d^{+}(v): v\in V\}$ and the \emph{minimum in-degree} $\delta^-(D)=\min\{d^{-}(v): v\in V\}$. We define the \emph{minimum semi-degree} of $D$ as $\delta^0(D)=\min\{\delta^+(D), \delta^-(D)\}$ and the \emph{minimum degree} $\delta(D)=\min_{x\in V}\{d(x): d(x)=d^+(x)+d^-(x)\}$.

\smallskip

All paths in digraphs refer to directed paths. We define the number of arcs of a path as its \emph{length} and a \emph{$k$-path} refers to a path of order $k$. We often represent the $k$-path $P$ as $v_1\cdots v_k$ when $V(P)=\{v_1, \ldots, v_k\}$ and call $v_1v_2$ and $v_{k-1}v_k$ the \emph{first end-arc} and the \emph{last end-arc} of $P$, respectively. A $4$-path $abcd$ is called as a \emph{reverse square $4$-path} if it further satisfies that $ca, db\in A(D)$. We say a reverse square path of order $k$ to be a \emph{reverse square  $k$-path}. We can similarly define reverse square $k$-cycles.

\smallskip

For any vertex $v \in V(D)$, we say that a reverse square $4$-path $abcd$ is an \emph{absorber} of $v$ it can be extended by absorbing vertex $v$ to a reverse square $5$-path $abvcd$. We also say the reverse square $4$-path $abcd$ absorbs $v$ if it is an absorber of $v$. Clearly, in an absorber $abcd$, the vertices $a$ and $c$ are out-neighbours of $v$, $b$ and $d$ are in-neighbours of $v$, and $c$ is an in-neighbour of $b$. For two reverse square paths $P=ab\cdots cd$ and $Q=cd\cdots pq$ with $V(P)\cap V(Q)=\{c,d\}$, we denote the concatenated path as $P\circ Q$. This definition can be extended naturally to more than two paths.

For a positive integer $t$, simply write $\{1, \ldots , t\}$ as $[t]$. Throughout this paper, the notation $0<\beta\ll\alpha$ is used to make clear that $\beta$ can be selected to be sufficiently small corresponding to $\alpha$ so that all calculations required in our proof are valid.
\subsection{Tools}
\begin{theorem}{\rm \cite{Czygrinow}} Suppose that $D$ is a digraph with $\delta(D)\geq(4|V(D)|-3)/3$, and let $c \geq 0$ and $t \geq 1$ be integers with $c+t = \lfloor \frac{|V(D)|}{3}\rfloor$. Then $D$ contains $c$ cyclic triangles and $t$ transitive triangles and they are disjoint. \end{theorem}
The following result is easily obtained from the above theorem.
\begin{corollary} \label{triangle} Every digraph with $\delta^0 (D) \geq 2|D|/ 3$ contains $\lfloor \frac{|V(D)|}{3}\rfloor-1$ disjoint cyclic triangles and one transitive triangle. \end{corollary}
\begin{lemma}\label{Chernoff} {\rm \cite{Janson}}
Let $X$ be a random variable with the expectation $\mathbb{E}X$, and let $a$ be any real number with $0<a<3/2$. Then the following statements hold.\\
$(1)$ \emph{(}Chernoff's inequality\emph{)} $\mathbb{P}(|X-\mathbb{E}X|>a\mathbb{E}X)<2e^{-\frac{a^2}{3}\mathbb{E}X}$.\\
$(2)$ \emph{(}Markov's inequality\emph{)}
$\mathbb{P}(X\geq a)\leq\frac{\mathbb{E}X}{a}$.
\end{lemma}
%
%
%
%

In the proof of Theorem \ref{main}, we need the Regularity lemma and the Blow-up lemma for digraphs. The interested readers can refer to \cite{Komlos1999,Komlos1996} for a survey on the Regularity Lemma and the Blow-up lemma. Before giving the statement of these lemmas, we need some definitions. Let $G=(A, B)$ be a bipartite graph. The \emph{density} of $G=(A, B)$ is
$$d_G(A, B) :=\frac{e_G(A, B)}{|A||B|}.$$
We often write $d(A, B)$ if this is unambiguous. Given $\epsilon > 0$, we say that $G=(A, B)$ is \emph{$\epsilon$-regular} if for all subsets $X \subseteq A$ and $Y \subseteq B$ with $|X| > \epsilon|A|$ and $|Y | > \epsilon|B|$ we have that $|d(X, Y )-d(A, B)| < \epsilon$. Furthermore, for any real number $d$ with $0\leq d\leq1$, we say that $G$ is \emph{$(\epsilon, d)$-super-regular} if it is $\epsilon$-regular, and satisfies $d_G(a) \geq (d -\epsilon)|B|$ for all $a \in A$ and $d_G(b) \geq (d - \epsilon)|A|$ for all $b \in B$.

The Regularity lemma for digraphs, called as the Diregularity lemma, was proved by Alon and Shapira \cite{Alon}.
We will use the degree form of the Diregularity lemma which can be easily derived from the standard version, and so we omit its proof.

\begin{lemma} \label{regular}{\rm (}Degree form of the Diregularity lemma{\rm )} For every $\epsilon \in (0, 1)$ and any integer $M^\prime $ there are integers $M=M(\epsilon, M^\prime)$ and $n_0=n_0(\epsilon, M^\prime)$ such that if $D$ is a digraph on $n \geq n_0$ vertices and $d$ is any real number with $0\leq d\leq1$, then there is a partition of $V(D)$ into $V_0, V_1,\ldots,V_k$ with $M^\prime \leq k \leq M$ such that\\
 $(1)$ $|V_0| \leq \varepsilon n$, and \\
 $(2)$ $|V_1|=\cdots = |V_k| =: m$,\\
  and a spanning subdigraph $D^\prime$ of $D$ such that\\
$(3)$ for all $i = 1,\ldots,k$ the digraph $D^\prime [V_i]$ is empty,\\
$(4)$ $d^+_{D^\prime} (x) > d^+_D(x)-(d + \varepsilon)n$ for all vertices $x \in D$, and $d^-_{D^\prime} (x)>d^-_D(x)-(d + \varepsilon)n$ for all vertices $x \in D$, and\\
$(5)$ for all $1 \leq i, j \leq k$ with $i \neq j$ the bipartite graph whose vertex classes are $V_i$ and $V_j$ and whose edges are all the $V_i-V_j$ edges in $D^\prime$ obtained by deleting their directions is $\varepsilon$-regular and has density either $0$ or density at least $d$.
\end{lemma}
\emph{Remark} $1.$ In Lemma \ref{regular}, the vertex sets $V_1,\ldots,V_k$ are called as \emph{clusters}, and $V_0$ is called as the \emph{exceptional set}. Also, the spanning subdigraph $D^\prime$ is \emph{the pure digraph} with parameters $\epsilon, d$. In particular, the last condition in Lemma \ref{regular} says that all pairs of clusters are $\epsilon$-regular in both directions (but possibly with different densities).

$2.$ \emph{The reduced digraph} $R$ with parameters $\epsilon, d$ is the digraph with the vertex set $[k]$,  in which $ij$ is an arc if and only if the bipartite graph with vertex classes $V_i$ and $V_j$ whose edges are all the $V_i-V_j$ edges in $D^\prime$ is $\epsilon$-regular and has density at least $d$. That is, if $D^\prime$ is the pure digraph, then $ij$ is an arc in $R$ if and only if there is a $V_i-V_j$ edge in $D^\prime$. It is easy to see that the reduced digraph $R$ obtained from Lemma \ref{regular} has the minimum semidegree
\begin{equation}\label{degree-R}
  \delta^0(R) \geq (\delta^0(D)/|D|-d-2\epsilon)|R|.
\end{equation}
\begin{lemma} {\rm \cite{Komlos}}\label{blowup} {\rm (}Blow-up lemma{\rm )} For any graph $F$ with the vertex set $[k]$, and any positive numbers $d$ and $\Delta$, there is a positive real $\eta_0 = \eta_0(d, \Delta, k)$ such that the following holds for all positive numbers $l_1,\ldots, l_k$ and all $0 < \eta \leq \eta_0$. Let $F^\prime$ be the graph obtained from $F$ by replacing each vertex $i \in F$ with a set $V_i$ of  $l_i$ new vertices and joining all vertices in $V_i$ to all vertices in $V_j$ whenever $ij$ is an edge of $F$. Let $G^\prime$ be a spanning subgraph of $F^\prime$ such that for every edge $ij \in F$ the graph $(V_i, V_j )_{G^\prime}$ is $(\eta, d)$-super-regular. Then $G^\prime$ contains a copy of every subgraph $H$ of $F^\prime$ with $\Delta(H) \leq \Delta$.\end{lemma}

In order to apply Lemma \ref{blowup}, it is sufficient to ensure that all the arcs in a specific oriented subgraph of the reduced digraph $R$ correspond to $(\epsilon, d)$-super-regular pairs of clusters. This is guaranteed by the following result from \cite{Kelly}.

\begin{lemma} {\rm \cite{Kelly} } \label{superregular}  Let $M^\prime$, $n_0$, $l$ be integers, and let $\varepsilon, d$ be positive constants such that $1/n_0 \ll 1/M^\prime \ll \varepsilon \ll d \ll 1/l$. Suppose that $G$ is an digraph of order at least $n_0$, and $R$ is the reduced digraph, and $G^\prime$ is the pure digraph obtained by applying Lemmas \ref{regular} with parameters $\varepsilon, d$ and $M^\prime$ to $G$.
Suppose that $G^\ast$ is an oriented graph obtained from $G^\prime$ by deleting all $V_i-V_j$ edges for some pairs.
Let $R^\prime$ be an oriented subgraph of $R$ with $\Delta(R^\prime) \leq l$.
Let $H$ be the underlying graph of $G^\ast$. Then one can delete $2l\varepsilon|V_i|$ vertices from each cluster $V_i$ to obtain subclusters $V_i^\prime \subset V_i$ in such a way that $H$ contains a subgraph $H_{R^\prime}$ whose vertex set is the union of all the $V_i^\prime$ and such that \\
$(1)$ $(V_i^\prime, V_j^\prime )_{H_{R^\prime}}$ is $(\sqrt{\varepsilon}, d- 4l\varepsilon)$-super-regular whenever $ij \in E(R^\prime)$.\\
$(2)$ $(V_i^\prime, V_j^\prime )_{H_{R^\prime}}$ is $\sqrt{\epsilon}$-regular and has density $d-4l\epsilon$ whenever $ij \in E(R)$.
\end{lemma}

\section{Proof of Theorem \ref{main}}

In this section, we call reverse square $k$-paths (reverse square $k$-cycles) as $k$-paths ($k$-cycles) for simplicity.

\subsection{Main lemmas}
In this subsection, let $\gamma$ be any real with $\gamma\ll 1$, and let $D$ be any digraph with $\delta^0(D)\geq(2/3+\gamma)n$. The following lemma asserts that any two disjoint arcs can be connected by a short directed reverse square path.

\begin{lemma} \label{connnect} {\rm (}Connecting Lemma{\rm )} Let $D$ be the digraph as showed in Theorem \ref{main}, and let $\gamma$ be a real with $\gamma\ll 1/6$. For every two disjoint arcs of $D$, there is a $k$-path with $k \leq 4/\gamma$ connecting them.
\end{lemma}

\begin{proof}
Let  $ab$ and $cd$ be two disjoint arcs in $D$. In the following,
we first build a out-cascade structure as follows. We construct vertex sets $X_0,X_1,\ldots, X_i, \ldots$,  and bipartite digraphs $G_1(X_0, X_1), G_2(X_1, X_2),\ldots, G_i(X_{i-1}, X_i), \ldots$. Let $X_0=\{b\}$, $X_1=\{x:xa,bx \in A(D)\}$ and $A(G_1)=\{ bx :x\in X_1\}$. Note that $|X_1| \geq (1/3 +2\gamma)n$. Further, let
\begin{equation*}
\begin{split}
&X_2^\prime=\{y:\ \exists\ x \in X_1 \ \text{such that}\ xy,yb \in A(D)\},\\
&A(G_2^\prime)=\{ xy:\ x\in X_1, y\in X_2^\prime \ \text{and}\ xy,yb \in A(D)\}\ \mbox{and} \\
&G_2^\prime:=G_2^\prime(X_1,X_2^\prime).
\end{split}
\end{equation*}
Then for every $xy \in G_2^\prime$ and $y\neq a$, there exists a 4-path $abxy$ in $D$. Furthermore, $d^+_{G_2^\prime}(x) \geq (1/3 +2\gamma)n$ for each $x \in X_1$. Let
\begin{center}
$X_2^0=\{ y \in X_2^\prime : d^-_{G_2^\prime}(y) < \sqrt{n} \}$, $X_2=X_2^\prime \setminus X_2^0$ and $G_2= G_2^\prime[X_1 \cup X_2]$.
\end{center}
Note that $(1/3 +2\gamma)n|X_1| \leq |A(G_2^\prime)| \leq |X_1||X_2|+\sqrt{n}|X_2^0| \leq n^{\underline{}\frac{3}{2}} +|X_1||X_2|$. This implies that $|X_2| \geq n/3$.

Next, assume that we already have constructed $X_0,X_1,\ldots,X_j$ and $G_1,\ldots,G_j$ where $j \geq 2$. To construct $X_{j+1}$ and $G_{j+1}$, we consider an auxiliary bipartite digraph $B_y^j$ between the in-neighbours of $y$ in $G_j$ and all vertices of $V(D)$ for each $y \in X_j$, where $A(B_y^j)=\{zx : x\in N^-_{G_j}(y),z\in V(D)\ \text{and}\ zx,yz\in A(D)\}$. Define $G_{j+1}^\prime$ the bipartite digraph between $X_j$ and $X_{j+1}^\prime$, where
\begin{equation*}
\begin{split}
X_{j+1}^\prime&=\{z \in V(D) : \exists\ y \in X_j \ \text{such that}\ d_{B_y^j}^+(z) \geq n^{\frac{1}{4}}\},\  \mbox{and}\\
A(G_{j+1}^\prime)&=\{yz\in A(D) :  y \in X_j, \ z \in X_{j+1}^\prime \}.
\end{split}
\end{equation*}
Furthermore, let
\begin{equation*}
\begin{split}
&X_{j+1}^0=\{ z \in X_{j+1}^\prime : d^-_{G_{j+1}^\prime}(z) < \sqrt{n} \}, \mbox{and}\\
&X_{j+1}=X_{j+1}^\prime \setminus X_{j+1}^0, G_{j+1}= G_{j+1}^\prime[X_j \cup X_{j+1}].
\end{split}
\end{equation*}
Notice that some of the sets $X_0,X_1,X_2,\ldots,X_{j+1}$ may intersect. Nevertheless, for the sake of our construction we treat them as disjoint by cloning the vertices as much as necessary. We call the structure consisting of the sets $X_0,X_1,X_2\ldots$ and the bipartite digraphs $G_1,G_2\ldots$, an \emph{$ab$-out-cascade}. We have to alter our construction for $j \geq 3$ and require $d_{B_y^j}^-(z) \geq n^{1/4}$ to ensure that we can obtain a legitimate reverse square path from any edge of $G_j$ going back to $ab$ as long as $j<n^{1/4}$, on which all vertices are distinct.

A vertex $y\in X_j$ is called \emph{heavy} if $d^-_{G_j}(y) \geq (1/3+\gamma)n$. We show that the following conclusion holds.

\begin{claim} \label{conn1} There exists an integer $j \leq j_0=\lceil \frac{1}{\gamma}\rceil +1$ such that $X_j$ contains at least one heavy vertex. \end{claim}

\begin{proof}
First, we show that for $j \geq 2$ and for every $y \in X_j$ the out-degree of $y$ in $G_{j+1}^\prime$ is at least $(1/3+2\gamma)n-n^{\frac{3}{4}}$. Let $s$ be the number of vertices $z \in V(D)$ with $d^+_{B_y^j}(z) <n^{\frac{1}{4}}$. Then $sn^{\frac{1}{4}}+(n-s)|N_{G_j}(y)| \geq |A(B_y^j)| \geq |N_{G_j}(y)|(1/3+2\gamma)n.$ Since $|N_{G_j}(y)|= d^-_{G_j}(y) \geq \sqrt{n}$ and $s \leq n$, we obtain $$n-s \geq (1/3+2\gamma)n-\frac{sn^{\frac{1}{4}}}{|N_{G_j}(y)|} \geq (1/3+2\gamma)n-n^{\frac{3}{4}}. $$
Thus $d^+_{G_{j+1}^\prime}(y) \geq (1/3+2\gamma)n-n^{3/4}$. Note that the total number of arcs of $G^\prime_{j+1}$ incident to the vertices of $X_{j+1}^0$ is smaller than $n^{3/2}$.

To the contrary, we suppose that the vertex set $X_j$ does not contain any heavy vertex for each $j=2,\ldots,j_0$. So $|A(G_j)| < |X_j|(1/3+\gamma)n$. On the other hand, we have
$$|A(G_j)| \geq |X_{j-1}|[(1/3+2\gamma)n-n^{\frac{3}{4}}]-n^{\frac{3}{2}}.$$
Since $|X_1| \geq (1/3+2\gamma)n$, and by the fact that $(1-x)e^x \leq 1$ with $x= \frac{3\gamma}{1+6\gamma}$, we have that
$$|X_{j_0}|  \geq \frac{|A(G_{j_0})| }{(1/3+\gamma)n} \geq \frac{1+6\gamma}{1+3\gamma}|X_{j_0-1}|-O(n^{\frac{3}{4}}) > (\frac{1+6\gamma}{1+3\gamma})^{j_0-1} \frac{n}{3} > e^{\frac{3}{1+6\gamma}}\frac{n}{3}>n,$$
a contradiction.
\end{proof}

Analogously, we may consider the in-neighbours of $c$ and the out-neighbours of $d$ to build a cascade structure, defined \emph{$cd$-in-cascade.} We also obtain that, in the in-cascade ($X_j,G_j$), there exists an integer $j \leq j_0=\lceil \frac{1}{\gamma}\rceil +1$ such that $X_j$ contains at least one vertex with $d^+_{G_j}(y) \geq (1/3+\gamma)n$.

For the given two arcs $ab$ and $cd$, we consider the $ab$-out-cascade ($X_j^{(1)},G_j^{(1)}$) and $cd$-in-cascade ($X_j^{(2)},G_j^{(2)}$). For $i=1,2$, let $b_i \in X_{j_i}^{(i)}$ be a heavy vertex in the corresponding cascade, where $j_i^{(i)}\leq j_0$.

If $b_1=b_2:=b$, then $b$ has at least $(1/3+\gamma)n$ out-neighbours in $X_{j_2-1}^{(2)}$ by the definition of the heavy vertex. Since $\delta^0 (D) \geq (2/3+\gamma)n$, there exist two vertices $a_1 \in X_{j_1-1}^{(1)}$ and $a_2 \in X_{j_2-1}^{(2)}$ such that $a_1b,ba_2, a_2a_1\in A(D)$. When $b_1\neq b_2$, we have $|N^-(b_1)\cap N^+(b_2)| \geq (1/3+2\gamma)n$. Choose $w$ be a vertex in $N^-(b_1)\cap N^+(b_2)$. It is easy to see that there are $u_1 \in N^+(b_1)\cap N^-(w)$ and $u_2 \in N^+(w)\cap N^-(b_2)$ such that $u_2u_1 \in A(D)$. Clearly, $u_1$ has at least $\geq 2\gamma n$ out-neighbours in $N^-(b_1) \cup X_{j_1-1}^{(1)}$ and $u_2$ has at least $\geq 2\gamma n$ in-neighbours in $N^+(b_2) \cup X_{j_2-1}^{(2)}$.

By the definition of the $ab$-out-cascade, there is a $(j_1 + 3)$-path $P_1$ connecting $ab$ and $b_1u_1$ and, by the definition of the $cd$-in-cascade, there is a $(j_2 + 3)$-path $P_2$, which is disjoint with $P_1$, connecting $u_2b_2$ and $cd$. Then $P_1 \circ u \circ P_2$ is a $k$-path connecting $ab$ and $cd$ with $k=(j_1 + 3)+(j_2 + 3)+1 \leq 2(j_0 + 4) \leq 4/\gamma$ due to $\gamma \leq 1/6$.
\end{proof}

The absorbing lemma asserts that there is one `reasonably sized' reverse square path that possesses the property that any `reasonably sized' subset of vertices can be absorbed into it which constructs a longer path with the same end-vertices. For each $v \in V(D)$, we use $\mathcal{A}_v$ to represent the set of all absorbers of $v$.

\begin{lemma} \label{absorbing} {\rm ( }Absorbing Lemma{\rm )} There is an $l$-path $P_A$ in $D$ with $l\leq 20\gamma^3n$ such that for every subset $U \subset V(D)\setminus V(P_A)$ of cardinality at most $\gamma^7n$ there is a path $P_{AU}$ in $D$ satisfying $V(P_{AU})=V(P_A) \cup U$ and $P_{AU}$ has the same end-arcs as $P_A$.
\end{lemma}
\begin{proof}
We first show the following several claims.
\begin{claim} \label{ab1} For every $v \in V(D)$ there are at least $6 \gamma^3n^4$ absorbers of $v$. \end{claim}
\begin{proof}
By the lower bound of $\delta^0(D)$, we have $|N_D^-(v)|\geq(2/3+\gamma)n$. For each $b\in N_D^-(v)$, clearly $|N_D^+(v)\cap N_D^+(b)|\geq2(2/3+\gamma)n-n=(1/3+2\gamma)n$. Since $d^-_D(b) \geq (2/3 +\gamma)n$, we get that $|N^-_D(b) \cap N^+_D(b)\cap N^+_D(v)|\geq 3\gamma n$. Let $c$ be any vertex in $N^-_D(b) \cap N^+_D(b)\cap N^+_D(v)$. Similarly, by the lower bound of $\delta^0(D)$ again, we have that $|N^+_D(c)\cap N^+_D(v)|, |N^-_D(b)\cap N^-_D(v)|\geq(1/3+2\gamma)n$. Let $a$ be any vertex in $N^-(b)\cap N^+(c)\cap N^+(v)$, and finally let $d$ be an arbitrary out-neighbour of $c$ in $N^-(b)\cap N^-(v)$ that is different from $a$. Then $abcd$ is a $4$-path that absorbs $v$. Hence the number of absorbers of $v$ is at least $(2/3+\gamma)n\cdot 3\gamma n  \cdot 3\gamma n \cdot3\gamma n \geq 6 \gamma^3n^4$, which proves the claim.
\end{proof}
\begin{claim} \label{ab2} There exists a family $\mathcal{F}$ of at most $2\gamma^4n$ disjoint absorbers of vertices of $D$ such that for every $v \in V(D)$, $|\mathcal{A}_v \cap \mathcal{F}| > \gamma^7n$. \end{claim}
\begin{proof}
We first select a family $\mathcal{F}^\prime$ of $4$-sets at random by including each of $n(n-1)(n-2)(n-3)\sim n^4$ $4$-sets independently with probability $\gamma^4 n^{-3}$ (some of the selected $4$-sets may not be absorbers at all). Then we affirm the following conclusions.

\emph{$(1)$ With probability $1-o(1)$, as $n\rightarrow\infty$, $|\mathcal{F}^\prime|<2\gamma^4n$ and $|\mathcal{A}_{uv}\cap\mathcal{F}^\prime|>4\gamma^7n$ for\\ $~~~~~~~~~~$ every vertex $v$ in $V(D)$.}

\emph{$(2)$ With probability at least $1/17$, as $n\rightarrow\infty$, there are at most $17\gamma^8n$ pairs of\\ $~~~~~~~~~~$ overlapping $4$-sets in $\mathcal{F}^\prime$.}

\smallskip

Conclusion $(1)$ can be obtained directly by using Chernoff's inequality. We further give the proof of $(2)$. Clearly, the expected number of intersecting pairs of $4$-sets in $\mathcal{F}^\prime$ is at most
\begin{equation*}
\begin{split}
n^4\times4\times4\times n^3\times(\gamma^4 n^{-3})^2=16\gamma^8n.
\end{split}
\end{equation*}
Let $X$ be the number of intersecting pairs of $4$-sets in $\mathcal{F}^\prime$. By Markov's inequality with $a=17\gamma^8n$, we can get that $\mathbb{P}(X\geq 17\gamma^8n)\leq\frac{\mathbb{E}X}{a}=\frac{16\gamma^8n}{17\gamma^8n}=16/17$. This implies that $(2)$ holds.

Hence, by $(1)$-$(2)$, there exists a random family satisfying properties $(1)$ and $(2)$ above with positive probability. For simplicity, we denote this family by $\mathcal{F}^{\prime}$. From $\mathcal{F}^{\prime}$ we obtain a subfamily $\mathcal{F}$ by deleting all $4$-sets that overlap other $4$-sets and all $4$-sets that are not absorbers at all. Then $\mathcal{F}$ consists of disjoint absorbers such that for each $v \in V(D)$ we have
\begin{equation*}
\begin{split}
|\mathcal{A}_v \cap \mathcal{F}| > 4 \gamma^7n-34\gamma^8n >\gamma^7n.
\end{split}
\end{equation*}\end{proof}

Let $f=|\mathcal{F}|$ and let $F_1, \ldots, F_f$ be the elements of $\mathcal{F}$, where $\mathcal{F}$ is the family as described in Claim \ref{ab2}. Since $F_i$ is an absorber, $F_i$ spans a $4$-path for every $i\in[f]$. In the following, we also use $F_i$ to represent the $4$-path. Then we will connect all these $4$-paths into one not too long path $P_A$. For this purpose, we will repeatedly apply Lemma \ref{connnect} to connect the last end-arc of $F_i$ and the first end-arc of $F_{i+1}$ by a short path for each $i\in[f-1]$.
\begin{claim} \label{ab3} There exists a path $P_A$ in $D$ of the form $P_A=F_1\circ P_1 \circ\cdots \circ F_{f-1} \circ P_{f-1} \circ F_f$, where each of the paths $P_1,\ldots, P_{f-1}$ has at most $8/\gamma$ vertices. \end{claim}
\begin{proof}
We show that, for each $i = 1,\ldots, f$, there exists a path $L_i$ in $D$ of the form $L_1 = F_1$, and $L_i=F_1\circ P_1 \circ \cdots \circ F_{i-1} \circ P_{i-1} \circ F_i$ for $i \geq 2$, where each of the paths $P_1,\ldots, P_{i-1}$ has at most $8/\gamma$ vertices.

There is nothing to prove for $i=1$. Assume the statement is true for some $1\leq i\leq f-1$. Let $ab$ be the last end-arc of $L_i$ and let $cd$ be the first end-arc of $F_{i+1}$. Denote $V_i=(V\setminus V(F\cup P_i)) \cup \{a,b,c,d\}$ and $D_i$ be the subdigraph of $D$ induced by $V_i$. Note that $|V(F\cup P_i)| < f(4+8/\gamma)< 20\gamma^3n$. We obtain that $\delta^0(D_i) \geq (2/3+\gamma/2)|D_i|$. Applying Lemma \ref{connnect} in $D_i$ on $ab$ and $cd$, we get a  path $P_i \subset D_i$ of length at most $8/\gamma$, with $ab$ is the first end-arc and $cd$ is the last end-arc. Then we obtain a path of the form $L_{i+1}= L_i \circ P_i \circ F_{i+1}$. Thus $P_A:=L_f$ is the desired path.
\end{proof}
By Claim \ref{ab3}, there exists a  path $P_A$ in $D$ of the form $P_A=F_1\circ P_1 \circ\cdots \circ F_{f-1} \circ P_{f-1} \circ F_f$.  Obviously, $|P_A| <20\gamma^3n$. Notice that $P_A$ contains all absorbers in $\mathcal{F}$ and for every $v \in V(D)$, $|\mathcal{A}_v \cap \mathcal{F}| > \gamma^7n$. This implies that for any vertex set $U\subset V(D)\setminus V(P_A)$ of size at most $\gamma^7n$, we can insert all vertices of $U$ into $P_A$ one by one, each time using a new absorbing $4$-path. This completes the proof of the lemma.
\end{proof}

\begin{lemma} \label{reservior} {\rm(}Reservoir Lemma{\rm)} Let $D$ be the digraph described as in Theorem \ref{main}, and let $\gamma$ be any real with $\gamma\ll1$. For every subset $W \subset V(D)$, $|W| \leq \gamma n/4$, there exists a subset $R \subset V(D) \setminus W$ {\rm(}called a reservoir{\rm)} such that $|R| = \lceil \gamma^7n/2\rceil$ and for every $x \in V(D)$, $$d^+_R(x)\geq (2/3 + \gamma/2)(|R|+4)\ \text{ and }\  d^-_R(x) \geq (2/3 + \gamma/2)(|R|+4).$$
\end{lemma}
\begin{proof}
Let $r=\lceil \gamma^7n/2\rceil$.  We choose $R$ randomly out of all $\tbinom{n-|W|}{r}$ possibilities and apply the probabilistic method again. For each vertex $v$, the random variable $X^+_v$, counting the out-neighbours of $v$ in $R$, has the hypergeometric distribution with expectation $\mathbb{E}X^+_v$ satisfying
$$r \geq \mathbb{E}X^+_v \geq \frac{d^+(v)- |W|}{n -|W|} r \geq (\frac{2}{3} + \frac{3}{4}\gamma) r.$$
By Chernoff's bound, we have
$$\mathbb{P}(X^+_v < (\frac{2}{3}+\frac{1}{2}\gamma)(r+4))\leq \mathbb{P}(X^+_v\leq \mathbb{E}X^+_v-\frac{1}{4}\gamma r+2)\leq \exp\left(-\frac{\gamma^2 r}{33}\right).$$

Let $X^-_v$ be the random variable counting the vertices of $R$ which are in-neighbours of $v$. Analogously, we can get that
$$\mathbb{P}(X^-_v < (\frac{2}{3}+\frac{1}{2}\gamma)(r+4))\leq \mathbb{P}(X^-_v\leq \mathbb{E}X^-_v-\frac{1}{4}\gamma r+2)\leq \exp\left(-\frac{\gamma^2 r}{33}\right).$$
Let $A$ be the event that there exists a vertex $v\in V(D)$ such that $X^+_v < (\frac{2}{3}+\frac{1}{2}\gamma)(r+4)$ or $X^-_v < (\frac{2}{3}+\frac{1}{2}\gamma)(r+4)$.
We obtain that
$$\mathbb{P}(A)\leq \text 2n\cdot \exp\left(-\frac{\gamma^2 r}{33}\right)=2n\cdot \exp\left(-\frac{\gamma^9 n}{66}\right)=o(1).$$
Hence for sufficiently large $n$, the event $A$ does not occur for R with high probability. Then we can fix a choice of $R$ that meets the conditions.
\end{proof}

\begin{lemma} \label{Rconnect} {\rm(}Reservoir-Connecting Lemma{\rm)} Let $D$ be the digraph shown as in Theorem \ref{main}, and let $\gamma$ be any real with $\gamma\ll1$. For every two disjoint arc of $D$, $ab$ and $cd$, there is a $k$-path in $D[R \cup \{a, b, c, d\}]$ connecting $ab$ and $cd$ with $k \leq \frac{16}{\gamma}$. Furthermore, this statement remains true even if at most $\gamma^8n$ vertices of $R$ are forbidden to be used on this connecting path.\end{lemma}

\begin{proof} Let $D_1=D[R \cup \{a, b, c, d\}]$. Observe that $\delta^0(D_1) \geq (2/3 + \gamma/2)(|R|+4)-\gamma^8n \geq (2/3 + \gamma^\prime)(|R|+4)$, where $\gamma^\prime=\gamma/4$. So in $D_1$, we have almost the same degree condition as in $D$ with $\gamma=\gamma^\prime$. Thus there is a $k$-path in $D_1$, $k \leq \frac{16}{\gamma}$, which connects $ab$ to $cd$.
\end{proof}
We aim to find a (reverse square) cycle $C$ in $D$ that contains $P_A$ (obtained in Lemma \ref{absorbing}) as a subpath and covers all but at most $\gamma^7n/2$ vertices of $V(D)$. To achieve this, we first establish the following Path-Covering Lemma.
\begin{lemma} \label{pathcover} {\rm(}Path-Covering Lemma{\rm)} For every real $\gamma > 0$, there exists $n_0$ such that every digraph $D$ on $n>n_0$ vertices with $\delta^0(D) \geq (2/3+\gamma)n$ contains a family of at most $\gamma^{10}n$ vertex-disjoint paths, covering all but at most $\gamma^7n/2$ vertices.
\end{lemma}

\begin{proof}
Define the number $M^\prime$ and additional constants such that
$$1/n_0 \ll  1/M^\prime \ll \epsilon \ll d \ll \gamma^7/2 \ll 1.$$
Suppose that $D$ is a digraph on $n \geq n_0$ vertices with minimum semi-degree $\delta^0(D) \geq (2/3+\gamma)n$. We apply the Diregularity lemma (Lemma \ref{regular}) to $D$ with parameters $\epsilon^2, d+8\epsilon^2$ and $ M^\prime$ to obtain a partition $V_0, V_1,\ldots,V_k$ of $V(D)$ and a pure digraph $D^\prime$ with $M^\prime \leq k \leq M(\epsilon,M^\prime)$, $|V_1|=\cdots=|V_k|=m$ and $|V_0| \leq \epsilon^2 n$. Here we assume that $n_0 >M/\gamma^{10}$. Let $R$ be the reduced digraph with parameters $(\epsilon^2, d+8\epsilon^2)$. By inequality (\ref{degree-R}), we get that
$$\delta^0(R) \geq (2/3 +\gamma - d-8\epsilon^2 - 2\epsilon^2)|R| > (2/3 +\gamma/2)|R|.$$
By Corollary \ref{triangle}, there exist at least $\lfloor |R|/3 \rfloor-1$ cyclic triangles. Let $T$ be the union of these triangles. Obviously, $T$ is an oriented subgraph of $R$ with $\Delta(T) \leq 2$. Set $t=|T|=3(\lfloor |R|/3 \rfloor-1)$.

Next, we obtain the spanning oriented subgraph $D^\ast$ corresponding to $T$ from the pure digraph $D^\prime$ by deleting all those vertices that lie in clusters not belonging to $T$ as well as deleting all the $V_i-V_j$ arcs for all pairs $i,j$ with $ij \in A(R)\setminus A(T)$. Note that we add at most $5n/M^\prime$ vertices to the exceptional set $V_0$.

Applying Lemma \ref{superregular} with $G=D$, $G^\prime=D^\prime$, $ G^\ast =D^\ast$ and $R^\prime =T$, we could obtain an oriented subgraph $D^{\ast\ast}$ of $D^\ast$ such that each arc of $T$ corresponds to an $(\epsilon,d)$-super-regular pair of density $d$, by adding at most $4\epsilon^2n$ further vertices to the exceptional set $V_0$. Note that the new exceptional set now satisfies $|V_0| \leq \epsilon^2 n + 5n/M^\prime +4\epsilon^2n < \epsilon n$. Let $H^\prime$ be the union of $t/3$ disjoint  paths of length $3(1-8\epsilon) m$ and let $H$ be the underlying graph of $H^\prime$.

Consider the underlying graph $T^\prime$ of $T$ and fix additional constants $l_1=\cdots=l_t=m-8\epsilon m$.
We apply Lemma \ref{blowup} with $F=T^\prime$ and the underlying graph of $D^{\ast\ast}$. Note that $H$ is a subgraph of $F^\prime$ which described in Lemma \ref{blowup}. This shows that the underlying graph of $D^{\ast\ast}$ contains a copy of $H$. It is easy to check that $D^{\ast\ast}$ contains a copy of  $H^\prime$. Thus $D$ contains $t/3$ disjoint  paths of length $3(1-8\epsilon) m$ covering all but at most $\epsilon n$ vertices. Note that $\epsilon \ll \gamma^7/2$ and $t/3 \leq (M-5)/3$. Hence $D$ contains a family of at most $\gamma^{10}n$ vertex-disjoint paths, covering all but at most $\gamma^7n/2$ vertices.
\end{proof}
\subsection{Completion of Theorem \ref{main}}
Suppose that $\gamma < 1/400 $ is any positive real number, and let $P_A$ be an
absorbing $l$-path in $D$ with $l = |V(P_A)| \leq 20\gamma^3n$ that is obtained by Lemma \ref{absorbing}. Applying Lemma \ref{reservior} to $D$ with $W = V(P_A)$, we get that a reservoir set $R \subset V(D-P_A)$ of cardinality $\gamma^7n/2$ with the property described in Lemma \ref{reservior}. For convenience, let $D^\prime = D[V \ (V(P_A) \cup R)]$ and $n^\prime=|D^\prime|$. Clearly, $\delta^0(D^\prime) \geq (2/3+\gamma_1)n^\prime$, where $ n \geq n^\prime \geq n-20\gamma^3n-\gamma^7n/2$ and
$\gamma >\gamma_1>\gamma -20\gamma^3-\gamma^7/2.$

\smallskip

Furthermore, using Lemma \ref{pathcover} to $D^\prime$, we can obtain a family of at most $\gamma_1^{10}n^\prime$ disjoint  paths, say $P_1,\ldots,P_p$, covering all but at most $\gamma_1^7n^\prime/2$ vertices. Let $T$ be the set of uncovered vertices. So $|T| \leq \gamma_1^7n^\prime/2 \leq \gamma^7n/2$.

\smallskip

Next our goal is to show that there exists a  cycle $C$ in $D$ such that $P_A$ is contained in $C$, $|V(D-C)| \leq \gamma^7n$ and $|V(C)\cap R| \leq \gamma^8n$. This will prove the theorem. Indeed, for the vertex subset $U=T \cup (R \setminus V(C))$, we have $U \subset V(D-P_A)$ and $|U|=|T|+|R\setminus V(C)| \leq \gamma^7n/2+\gamma^7n/2 =\gamma^7n/2$. By Lemma \ref{absorbing}, there is a  path $P_{AU}$ in $D$ with $V(P_{AU}) = V(P_A) \cup U$ and $P_{AU}$ has the same end-arcs as $P_A$. Clearly, $U=D-C$. Hence, we could absorb all vertices outside $C$, and then obtain the reverse square of a Hamiltonian cycle.

\smallskip

To construct the  cycle $C$, we connect all  paths $P_1,\ldots, P_p$ as well as the  path $P_A$ by successively applying
Lemma \ref{Rconnect}. Notice that new vertices of the connecting paths will be entirely contained in the set $R$. In what follows, we will show the procedure of connecting together all paths into one  cycle $C$.

Set $P_{p+1}=P_A$. We show that for each $i\in[p+1]$, there exists a  path $L_i$ in $D$ of the form $L_1 =P_1$ and for $i \geq 2$,
$$L_i = P_1 \circ Q_1 \circ \cdots \circ P_{i-1} \circ Q_{i-1} \circ P_i$$
where each of the  paths $Q_1,\ldots,Q_{i-1}$ has at most $16/\gamma$ vertices and belongs to $R$. There is nothing to prove for $i = 1$. Assume the statement is true for some $1 \leq i \leq p$. Let $ab$ be the last end-arc of $L_i$ and let $cd$ be the first end-arc of $P_{i+1}$. Note that
$$\bigcup_{j=1}^{i-1} |V(Q_j)| <\gamma(i-1)/16< 16\gamma^9n < \gamma^8n.$$
Applying Lemma \ref{Rconnect} with $ab$ and $cd$, we get a  path $Q_i$ of order at most $16/\gamma$, with $ab$ is the first end-arc and $cd$ is the last end-arc, such that $Q_i\setminus \{a,b,c,d\} \subset R$ and $Q_i \cap L_i = \{a,b,c,d\}$. Then we obtain a path of the form
$L_{i+1} = L_i \circ Q_{i} \circ P_{i+1}.$
Thus we could obtain path $L_{p+1}$ which contains an absorbing  path $P_A$ and satisfies $|V(L_{p+1}) \cap R | < 16\gamma^9n < \gamma^8n $.

\smallskip

To obtain the  cycle $C$, let $ab$ be the last end-arc of $L_{p+1}$ and $cd$ be the first end-arc of $L_{p+1}$. Applying Lemma \ref{Rconnect} with $ab$ and $cd$ again, we obtain a  path $Q$ of order at most $16/\gamma$ connecting $ab$ and $cd$ such that $Q\setminus \{a,b,c,d\} \subset R$. Thus we get desired  cycle
$$C=P_1 \circ Q_1 \circ \cdots \circ P_{p+1} \circ Q \circ P_1$$
that contains $P_A$ as a segment and satisfies $|V(C)\cap R| \leq \gamma^8n$.
Recall that $|R| = \gamma^7n/2$ and $|T| \leq \gamma^7n/2$. There are at most $\gamma^7n$ vertices left outside $C$. In fact, we have
$$|V(D-C)| = |T \cup R \ V(C)| \leq |T| + |R| \leq \frac{1}{2}\gamma^7n+\frac{1}{2}\gamma^7n = \gamma^7n.$$
Thus this completes the proof of Theorem \ref{main}.

\section{Concluding remarks}
To conclude the papaer, we pose the following problem.
\begin{problem}
Does every digraph on $n$ vertices with $\delta^0(D)\geq 2n/3$ contain the (reverse) square of Hamiltonian cycles?
\end{problem}
Also, the cycle factor is also a natural extension of the Hamiltonian cycle. We conjecture that for any digraph $D$ the lower bound $2|V(D)|/3$ on the semi-degree of $D$ can guarantee all possible cycle factors. This is asymptotically correct for the cases $n_i=3$  ($i\in[k]$) and $n_i=4$ ($i\in[k]$) due to a result of Czygrinow et. al \cite{Czygrinow}.
\begin{conjecture}\label{conj2}
Every digraph of order $n$ with $\delta^0(D)\geq2n/3$ contains all possible cycle factors. That is, for any positive integer partition  $n=n_1+\cdots+n_k$ with $n_i\geq 3$ for each $i$, digraph $D$ has $k$ disjoint cycles of length $n_1,\ldots,n_k$, respectively.\end{conjecture}

The following example shows that the minimum semi-degree condition in Conjecture \ref{conj2} would be best possible if true. Let $X$ and $Y$ be two disjoint vertex sets of size $2k-1$ and $k+1$, respectively. Suppose $D$ is a digraph with vertex set $X\cup Y$ and arc set $\{xy: x\in X, y\in X\cup Y\}\cup \{yx: x\in X, y\in Y\}$.  Clearly, $\delta^0(D)=2k-1=2n/3-1$, but it contains no $k$ disjoint 3-cycles.

\nocite{*}
\bibliographystyle{abbrvnat}
\bibliography{cite}

\end{document}